\documentclass{amsart}
\usepackage{amssymb} 
\usepackage{amsmath} 
\usepackage{mathrsfs} 
\usepackage{eufrak}
\usepackage{amscd}
\usepackage{amsbsy}
\usepackage{comment}
\usepackage[matrix,arrow]{xy}
\usepackage{hyperref}
\usepackage{mathtools}
\usepackage{esint}
\usepackage[letterpaper, margin=1in]{geometry}

\newcommand{\N}{{\mathbb N}}

\newcommand{\R}{{\mathbb R}}

\newcommand{\D}{{\mathbb D}}
\newcommand{\sphere}{{\mathbb S}}

\newcommand{\scM}{{\mathscr M}}

\newcommand{\inv}{^{-1}}

\newcommand{\us}{\underset}
\newcommand{\ol}{\overline}
\newcommand{\supp}{\text{supp}}
\newcommand{\ep}{\epsilon}

\newcommand{\tr}{\text{tr}}
\newcommand{\RNum}[1]{\uppercase\expandafter{\romannumeral #1\relax}}

\newcommand{\be}{\begin{equation}}
\newcommand{\ee}{\end{equation}}
\newcommand{\ba}{\begin{align*}}
\newcommand{\ea}{\end{align*}}

\newtheorem{lem}{Lemma}[section]

\newtheorem{theorem}{Theorem}[section]

\newtheorem{remark}{Remark}[section]
\newtheorem{lemma}{Lemma}[section]

\theoremstyle{definition}

\def\mod#1{{\ifmmode\text{\rm\ (mod~$#1$)}
\else\discretionary{}{}{\hbox{ }}\rm(mod~$#1$)\fi}}
\begin{document}

\bibliographystyle{amsplain}

\author{Albert Chau$^1$}
\address{Department of Mathematics,
The University of British Columbia, Room 121, 1984 Mathematics
Road, Vancouver, B.C., Canada V6T 1Z2} \email{chau@math.ubc.ca}

\thanks{$^1$Research
partially supported by NSERC grant no. \#327637-06}

\author[A. Martens]{Adam Martens}
\address{Department of Mathematics, The University of British Columbia, 1984 Mathematics Road, Vancouver, B.C.,  Canada V6T 1Z2.  Email: martens@math.ubc.ca. }

\title{Exterior Schwarzschild initial data for degenerate apparent horizons}

\maketitle

\vspace{-20pt}

\begin{abstract}

In this note we show that if $g$ is a smooth Riemannian metric on $\mathbb{S}^2$ such that the first eigenvalue of the operator $L_g:=-\Delta_g +K_g$ satisfies $\lambda_1(L_g)=0$ then $(\mathbb{S}^2, g)$ arises as an apparent horizon in an asymptotically flat initial data set with ADM mass arbitrarily close to the associated Hawking mass $\sqrt{\text{area}(\mathbb{S}^2, g)/16\pi}$.  In particular, this determines the Bartnik quasilocal mass (introduced by Bartnik \cite{Bartnik} in 1989) associated with $(\mathbb{S}^2, g)$ in this setting.  We prove these by modifying the construction of Mantoulidis-Schoen \cite{MS} who proved the same results in the case $\lambda_1(L_g)>0$.   It follows that $\lambda_1(g)\geq 0$ is necessary and sufficient for $(\mathbb{S}^2, g)$ to arise from an apparent horizon in an asyptotically flat space-time under the dominant energy condition and in the time symmetric setting, and that the Bartnik mass of the horizon is $\sqrt{\text{area}(\mathbb{S}^2, g)/16\pi}$.



\end{abstract}

\section{Introduction}

 An apparent horizon in a 4-dimensional asymptotically flat space-time $N^4$ is an outermost closed 2-dimensional surface $\Sigma$ in an initial data set $M^3$ with the property that the future directed outgoing null normal field $l_+$ on $\Sigma$ has zero divergence, where the divergence is defined as the scalar trace of the second fundamental form of $\Sigma \subset N^4$ relative to $l_+$.  This property in particular implies that light rays from the surface do not move outward.  Under the dominant energy condition and in the time symmetric setting, both of which will be assumed throughout this paper, an apparent horizon can be characterized as an embedded stable minimal 2-sphere  $\mathbb{S}^2$ in an asymptotically flat 3-dimensional Riemannian manifold $(M^3, G)$ with non-negative scalar curvature $R_G\geq 0$ \cite{Gibbons, Haw, Haw2, HI}.  In particular, the stability condition implies
\begin{equation}\label{stab1}
\int_{\mathbb{S}^2} (-\Delta_g u +K_g u)u d\mu_g \geq \frac{1}{2}\int_{\mathbb{S}^2} (R_G+\|\rho\|^2)u^2 d\mu_g
\end{equation}
for all smooth functions $u$ on $\mathbb{S}^2$ where $K_g$ is the Gauss curvature of the induced metric $g$ on $\mathbb{S}^2$, and $\rho$ is the second fundamental form on $\mathbb{S}^2 \subset (M, G)$.   Since $R_G\geq 0$ it follows that the operator $ L_g:=-\Delta_g +K_g$ is non-negative (has non-negative first eigenvalue $\lambda_1(g):=\lambda_1(L_g)$) on any such apparent horizon $(\mathbb{S}^2, g)$. We summarize this as

$$AH \subset \overline{\scM_+}$$
where 
$$AH:=\{ \begin{text}smooth \,\,metrics \,\, g \,\, on \,\, \mathbb{S}^2 \,\, arising \,\, from \,\, an \,\, apparent\,\,horizon\end{text} \},$$
$$\scM_+:=\{ \begin{text}smooth \,\,metrics \,\, g \,\, on \\, \mathbb{S}^2 \,\,with\,\, \lambda_1(g)>0 \end{text}\}, \text{ and}$$
$$\overline{\scM_+}:=\{ \begin{text}smooth \,\,metrics \,\, g \,\, on \\, \mathbb{S}^2 \,\,with\,\, \lambda_1(g) \geq 0\end{text}\}.$$

 While the space of apparent horizons $AH$ is defined extrinsicly in terms of an ambient manifold $M$, the space $\overline{\scM_+}$ is defined purely instrinsicly on $\mathbb{S}^2$.   In this sense, the opposite inclusion would provide a complete instrinsic characterization of $AH$, though this inclusion is far from obvious.  Mantoulidis-Schoen \cite{MS} proved the inclusion  $\scM_+ \subset AH $.  More precisely, they proved that if $g\in \scM_+$ then $(\mathbb{S}^2, g)$ admits an admissible extension with non-negative scalar curvature in which the boundary $(\mathbb{S}^2,g)$ is outermost minimal and stable as in condition \eqref{stab1}, and thuss $(\mathbb{S}^2, g)$ is indeed an apparent horizon.  Of equal importance, they were actually able to construct such an admissible extension with  ADM mass prescribed arbitrarily close to $\sqrt{\text{area}(\mathbb{S}^2, g)/16\pi}$ from above, thus implying that the Bartnik quasilocal mass associated with $(\mathbb{S}^2, g)$ is $\sqrt{\text{area}(\mathbb{S}^2, g)/16\pi}$ by means of the Riemannian-Penrose inequality, proved by Huisken and Ilmanen in \cite{HI} and in full generality by Bray in \cite{Bray}. 


 In this paper, we extend the above results  in  \cite{MS} to include the degenerate case of when  $\lambda_1(g)=0.$ Namely, we prove the  inclusion   $\overline{\scM_+}\subset AH $ thus charachterizing the space of apparent horizons, under the assumptions of the dominant energy condition and the time symmetry setting. 

\begin{theorem}\label{T}

The inclusion $\overline{\scM_+}\subset AH$ holds (and thus $\overline{\scM_+}= AH$).  Moreover, the Bartnik quasilocal mass of any apparent horizon $(\mathbb{S}^2, g)$ is $\sqrt{\text{area}(\mathbb{S}^2, g)/16\pi}$.

 More precisely, if $g\in \ol{\scM_+}$, then for any $m>\sqrt{\text{area}(\mathbb{S}^2, g)/16\pi}$ there exists a smooth metric $G(p, t)$ on the manifold with boundary $M^3= \mathbb{S}^2\times [0, \infty)$ such that $G(\cdot, 0)\equiv g$  and:  \\

\begin{itemize}
\item [(i)] $(M^3,G)$ has non-negative scalar curvature which vanishes on $\partial{M^3} $,
\item [(ii)] The second fundamental form $\rho_{ij}$ of $ \partial{M^3} $ in $(M^3,G)$ is identically zero,
\item [(iii)] the foliating spheres $\mathbb{S}^2\times\{t\}$ are strictly mean convex in $(M^3, G)$ for all $t>0$, and
\item  [(iv)] for some $T>2m$, the metric $G(p,t)$ is equal to the standard mass-$m$  Riemannian-Schwarzschild metric
$$g_{S, m}(p, t)=t^2g_*(p)+\left(1-\frac{2m}{t}\right)\inv dt^2$$ for $t>T$ where $g_*$ is the standard round metric on $\mathbb{S}^2$. 
 \end{itemize}
\end{theorem}

\begin{remark}
We refer to Lemma \ref{L1.3} (ii) for the sense of mean convexity we are using here.  In particular, by (iii) and the maximum principle, it follows that the interior of $(M^3,G)$ does not contain any compact smooth minimal surfaces, and in this sense $ \partial{M^3} $  will be the outermost minimal surface in $(M^3,G)$.
\end{remark}

 The proof is sketched as follows.  The first step is to construct a so called ``collar" which extends the metric $g$ on $\mathbb{S}^2\times \{0\}$ to a warped product metric $\gamma$ on $\mathbb{S}^2 \times [0, 1]$ satisfying certain properties (see Lemma \ref{L1.3}).  The construction closely follows that in \cite[\S1]{MS}, except for a key difference in our choice of the warping factor which will allow us to begin the construction for arbitrary metrics in $\overline{ \scM_+}$.  The details of this are carried out in \S 2.  From this point, the construction in \cite[\S2]{MS} can be envoked, providing a way to join the collar to an exterior Schwarzschild region resulting in a Riemannian metric $G$ on $\mathbb{S}^2\times [0, \infty)$ satisfying the conclusions of Theorem \ref{T}.  This is presented in \S 3.  Finally, in the Appendix we construct an explcit example of a smooth metric $g$ with $\lambda_1(g)=0$, demonstrating that the boundary of $\overline{\scM_+}$ is indeed non-empty.

\section{Extending to $\mathbb{S}^2\times[0, 1]$: a collar extension of $g$}
 We begin with the following.

\begin{lemma}\label{L1.2} For any $g\in \ol{\scM_+}$, there exists a smooth path of metrics $t\mapsto g(t)\in \ol{\scM_+}$, $t\in [0,1]$ with the following properties: \\
\begin{itemize}
\item [(i)] $g(0)=g$, 
\item [(ii)] $g(1)$ is round (constant positive curvature),
\item [(iii)] $\frac{d}{dt}g(t) \equiv 0$ for $t\in [1/2, 1]$, 
\item [(iv)] $\frac{d}{dt} dA_{g(t)}\equiv 0$ for all $t\in [0,1]$ where $dA_g$ denotes the area form of $g$, and
\item [(v)] $\lambda_1(g(t))\geq \alpha t$ for all $t\in [0,1]$ and some $\alpha>0$ depending only on $g(0)$. 
\end{itemize}
\end{lemma}

\begin{remark} \label{r1} When $g\in \scM_+$, this result was proved in  \cite[Lemma 1.2]{MS} where condition (v) was implicitly replaced with the stronger result $\lambda_1(g(t)) \geq c>0$ for some $c$ and all $t$.   When $g\in \partial \scM_+$ however, we have $ \lambda_1(g)=0$ and the estimate in (v) will be crucial for our constructions.
\end{remark}

\begin{proof}
Conditions (i)-(iv) were proved in  \cite[Lemma 1.2]{MS} and we only sketch the proof here, then highlight the modification needed to ensure the estimate in (v) when $g\in \partial \scM_+$.

 By uniformization we may write $g=e^{2w} g_*$ for some round metric $g_*$ with area $4\pi$.  Now fix a smooth decreasing function $\zeta: [0,1]\to [0,1]$ with $\zeta(0)=1$, $\zeta(1)=0$ and $\zeta\equiv 0$ on $[1/2, 1]$.   Then the proof of  \cite[Proposition 1.1]{MS} shows that $t\mapsto e^{2\zeta(t) w(x)}g_*$ is a smooth path in $\ol{\scM_+}$ from $g$ to the round metric $g_*$.  Let $h(t)=e^{2w \zeta(t)+2a(t) }g_*$ where $a(t)$ is smooth, $a(0)=0$ and 
$$a'(t)=-\zeta'(t)\fint_{\sphere^2} w(x) dA_{e^{2\zeta w}g_*}.$$
It is then shown in  \cite[Lemma 1.2]{MS} that the family $g(t)=\phi_t^* h(t)$ satisfies conditions (i)-(iv) in the Lemma where $\phi_t$ is the integral flow along the vector field $X_t$ satisfying div$_{h_t}X_t=-2(\zeta'(t)w+a'(t))$.

We now complete the proof of the Lemma by showing that  part (v)  holds when $g\in \partial \scM_+$ provided we prescribe  $\zeta'(0)<0$.  Note that $\lambda_1(g(t)) =\lambda_1(h(t))=e^{-2a(t)}\lambda_1(e^{2\zeta(t)w}g_*)$.   Thus, it is sufficient  to prove
\be\label{rho}
\lambda_1(e^{2\zeta(t)w} g_*)\geq \beta t
\ee for all $t$ and some $\beta >0$.  Moreover, since $\sigma(t):=e^{2\zeta(t)w} g_* \in \scM_+$ for all $t\in [c, 1]$ for any $c>0$ (see Remark \ref{r1}), we only need to prove that inequliaty \eqref{rho} holds for $0<t\ll1$ and some $\beta>0$.

We  prove inequality \eqref{rho} as follows.  Since $\zeta: [0,1]\to [0,1]$ is smooth, we can write $\zeta(t)=1+\zeta'(0)t+O(t^2)$. As explained in \cite[appendix]{MS}, there is a smooth positive function $u: [0,1]\times \sphere^2 \to \R_{>0}$ such that $u_t:= u(t, \cdot)$ is a first eigenfunction of $\sigma(t)$ with unit $L^2$ norm (with respect to the area form $dA_{\sigma(t)}$).   

 Now fix some $t\in [0, 1]$.  For any $s\in [0, 1]$, the formula for $K_{\sigma(s)}$ gives
$$\int |\nabla^{\sigma(s)} u_t|_{\sigma(s)}^2+K_{\sigma(s)} u_t^2 \; dA_{\sigma(s)}=\int |\nabla^{*} u_t|_{*}^2+ (1-\zeta(s) \Delta_* w) u_t^2 \; dA_{*}$$
which we view as a linear function $L_t(\zeta(s))$ of the variable $\zeta(s)$ and can thus be written as 
$$L_t(\zeta(s))=\zeta(s)L_t(\zeta(0))+(1-\zeta(s))L_t(\zeta(1))$$ for all $s\in [0, 1]$ where we have used the above properties of $\zeta$.  In particular, evaluating the above line at $s=t$  lets us estimate 
\begin{align*}
\lambda_1(\sigma(t))&=\int |\nabla^{\sigma(t)} u_t|_{\sigma(t)}^2+K_{\sigma(t)} u_t^2 \; dA_{\sigma(t)}\\&=
\zeta(t) \int |\nabla^{\sigma(0)} u_t|_{\sigma(0)}^2+K_{\sigma(0)} u_t^2\; dA_{\sigma(0)}+(1-\zeta(t)) \int |\nabla^{\sigma(1)} u_t|_{\sigma(1)}^2+K_{\sigma(1)} u_t^2\; dA_{\sigma(1)}\\&=
\zeta(t) \int |\nabla^{g} u_t|_{g}^2+K_{g} u_t^2\; dA_{g}+(1-[1+\zeta'(0)t+O(t^2)]) \int |\nabla^{g_*} u_t|_{g_*}^2+K_{g_*} u_t^2\; dA_{g_*}\\&\geq
\zeta(t) \lambda_1(g) \int u_t^2 \; dA_g+[-\zeta'(0)t+O(t^2)]) \lambda_1(g_* ) \int u_t^2 \; dA_{g_*}\\&\geq
[-4\pi \zeta'(0) \lambda_1(g_* ) \inf_{x} u_t(x)^2]t+O(t^2). 
\end{align*}

 In particular, recalling that $u: [0,1]\times\mathbb{S}^2\to \R_{>0}$ is smooth and positive we may conclude from the last line above that for $t$ sufficiently small we have $\inf_{x} u_t(x)^2\geq \inf_x u_0(x)^2/2>0$ and thus
$$
\lambda_1(g_t)\geq[-\pi\zeta'(0) \lambda_1(g_*) \inf_x u_0(x)]t=:\beta t.
$$
Recall that $\zeta'(0)<0$ and $\lambda_1(g_*)=1$ since $g_*$ is round. This completes the proof of the Lemma.

\end{proof}

 Now we fix some $g\in \partial \scM_+$ and consider the path $t\mapsto g(t)$ constructed above.  Fix some  smooth positive function $u: [0,1]\times \sphere^2\to \R_{>0}$ such that $u(t,\cdot)$ is a first eigenfunction for $L_{g(t)}$ with unit $L^2$ norm with respect to the area form $dA_{g(t)}$ (see \cite[appendix]{MS}).   \\

\begin{lemma}\label{L1.3} There exists $0<\ep_0\ll1$ and $A_0\gg1$ depending on $g(0)$ such that for all $0<\ep\leq \ep_0$, $A\geq A_0$, the topological cylinder $\sphere^2\times (0,1]=:\Sigma$ endowed with the metric 
$$
\gamma=(1+\ep t^2) g(t)+\Phi(t)^2 u(t,\cdot)^2 dt^2
$$
has the following properties:

\begin{itemize}
\item[(i)] $\gamma$ has positive scalar curvature which approaches zero uniformly as $t\to 0$.
\item[(ii)] the foliating spheres $\sphere^2\times \{t\}$ are mean convex for all $t>0$ in the sense: $H_t:=-tr_{\gamma} \rho >0$ for all $t>0$ where $\rho$ is the scalar second fundamental form of $\sphere^2\times \{t\}$ in $(\Sigma, \gamma)$ relative to the outward normal direction $\partial_t$, 
\item[(iii)] $\rho \to 0$ uniformly as $t\to 0$.
\end{itemize}
Here $\Phi(t): (0,1]\to \R_{>0}$ is defined as 
$$
\Phi(t) = \left\{
     \begin{array}{lr}
       \frac{A}{\sqrt{t}}& : t\in (0,1/4]\\
       \varphi(t) & : t\in (1/4,1/2]\\
       2A-1 & : t\in (1/2, 1]
     \end{array}
   \right.
$$
where $\varphi$ is a smooth, decreasing, convex function chosen so that $\Phi\in C^\infty((0,1])$.

\end{lemma}
\begin{remark}
When $g\in  \scM_+$ it was proved in  \cite[Lemma 1.3]{MS} that the metric $\gamma=(1+\ep t^2) g(t)+A^2 u(t,\cdot)^2 dt^2$ satisfies the same conclusions (i)-(iii) for a sufficiently large constant $A$.  Their proof relies on the fact that  $\lambda_1(g)>0$, and thus does not extend to the case $g\in  \partial \scM_+$.  We get around this by replacing this constant $A$ with our choice of $\Phi(t)$ above, and using the fact that $\lambda_1(g(t))\geq \alpha t$.
\end{remark}

\begin{remark}\label{singularity}
The presumed singularity at $t=0$ is superficial in the sense that changing to the new coordinate $s=\sqrt{t}$ on $\sphere^2\times (0,1/4)$ gives
$$
\gamma=(1+\ep s^4) g(s^2)+4A^2 u(s^2, \cdot)^2 ds^2
$$ 
which extends smoothly to $\sphere^2\times [0,1/4)$.    Moreover, since the mean curvature is coordinate invariant and continuous along the foliating spheres $\sphere^2\times\{s\}$, we obtain that the boundary sphere $\{s=0\}$ is minimal in $\sphere^2\times [0,1]$ relative to the extension by part (ii) of the Lemma.  
 In fact, Lemma \ref{L1.3} could have been stated and proved for this simpler parametrization as well, but we chose to use the parameter $t$ in the proof for ease of reference to \cite{MS} and in \S 3. 
\end{remark}

\begin{proof}[Proof of Lemma \ref{L1.3}] The proof of the Lemma is the same as the proof of  \cite[Lemma 1.3] {MS} except for part (i) where we deviate in our choice of the component $\Phi(t)$.  We provide the full details of all three parts of the proof for the readers convenience.

Write $h(t)=(1+\ep t^2) g(t)$, $v(t,x)=\Phi(t)u(t,x)$, and $\mu(t)=(1+\ep t^2)\inv \lambda(t)$ to simplify our notation to $\gamma=h(t)+v(t,x)^2 dt^2$ and $v(t,\cdot)$ now being an eigenfunction of $-\Delta_{h(t)}+K_{h(t)}$ with eigenvalue $\mu(t)$.  

Now fix any $t\in (0,1]$. The mean curvature of the sphere $\sphere^2\times \{t\}$ as a submanifold of $(\Sigma, \gamma)$ is $H_t=\frac{1}{2}\tr_{h(t)}(\rho)$ where $\rho=\langle N,\RNum{2}\rangle_\gamma$. Here $N=\frac{1}{v(t,\cdot)} \frac{\partial}{\partial t}$ is the (outward) unit normal and $\RNum{2}$ is the second fundamental form. To calculate this, let $E_1, E_2$ be a local coordinate frame field on $\mathbb{S}^2\times\{t\}$, which extend  naturally to  coordinate vector fields in the product space $\Sigma$. Then
\begin{align*}
\rho_{ij}
&=\langle N,\RNum{2}(E_i,E_j)\rangle_\gamma
=\gamma_{ab} N^a ((\nabla^{\gamma}_{E_i} E_j )^\perp)^b
=\gamma_{tt} N^t ((\nabla^{\gamma}_{E_i} E_j )^\perp)^t
\\&=\gamma_{tt}\frac{1}{v(t,\cdot)}(\Gamma_{\gamma})_{ij}^t
=\gamma_{tt}\frac{1}{v(t,\cdot)}(-\frac{1}{2}\gamma^{tt}\gamma_{ij;t})
=-\frac{1}{2v(t,\cdot)}h_{ij;t}.
\end{align*}
Therefore
$$
H_t=\frac{1}{2}\tr_{h(t)}(\rho)=-\frac{1}{4v(t,\cdot)} \tr_h(\dot{h})=-\frac{1}{4v(t,\cdot)} \tr_h(\dot{h})=-\frac{1}{v(t,\cdot)}\ep t(1+\ep t^2)^{-1}
$$
where in the last equality we have used that $\dot{h}=2\ep t g+(1+\ep t^2) \dot{g}$ and the fact that $\tr_g \dot{g}\equiv 0$ by Lemma \ref{L1.2} (iii).  In particular, the spheres $\sphere^2\times \{t\}$  are  strictly mean-convex relative to the outward normal direction since $u(t,\cdot)>0$.  Meanwhile, we also have $\rho_{ij}(x, t)\to 0$ uniformly on $\sphere^2\times \{t\}$ as $t\to 0$ since $\Phi(t)$ and thus $v(t, x)\to \infty$ uniformly as $t\to \infty$ while $h_{ij; t}$ is uniformly bounded for $t$ close to zero.  This establishes (ii) and (iii).

 To prove (i), we calculate
$$
\ddot{h}=2\ep  g+4\ep t\dot{g}+(1+\ep t^2) \ddot{g} \; \text{ and } \; \tr_h \ddot{h}=4\ep(1+\ep t^2)\inv+\tr_g \ddot{g}.
$$
The scalar curvature of the warped product metric $\gamma$ is
\begin{equation}
\begin{split}\label{Restimate}
R_\gamma&=2K_h-2v\inv \Delta_h v+v^{-2}\left[-\tr_h \ddot{h} -\frac{1}{4}(\tr_h \dot{h})^2+\frac{\partial_t v}{v}\tr_h \dot{h} +\frac{3}{4}|\dot{h}|^2_h\right]\\&=
2\mu+v^{-2}\left[-\tr_h \ddot{h} -\frac{1}{4}(\tr_h \dot{h})^2+\frac{\partial_t v}{v}\tr_h \dot{h} +\frac{3}{4}|\dot{h}|^2_h\right]\\&\geq
2\mu+v^{-2}\left[-\tr_h \ddot{h}+\frac{\partial_t v}{v}\tr_h \dot{h} \right]\\&=
2(1+\ep t^2)\inv \lambda+\Phi^{-2}  u^{-2}\left[-4\ep (1+\ep t^2)\inv-\tr_g \ddot{g}+4\ep t (1+\ep t^2)\inv \frac{\Phi\partial_t u+u \partial_t \Phi}{\Phi u}\right]\\&=
\Phi^{-2}(1+\ep t^2)\inv u^{-2}\left[2\Phi^2\lambda u^2-4\ep-(1+\ep t^2)\tr_g \ddot{g}+4\ep t\frac{\partial_t u}{u}+4\ep t\frac{\partial_t \Phi}{\Phi}\right]\\
&=:
\Phi^{-2}(1+\ep t^2)\inv u^{-2}\left[   I+II+III+IV+V \right].
\end{split}
\end{equation}

By the smoothness of the family $g(t)$, the definition of $\Phi$ and the fact that $\inf_{t,x} u^2 >0$ (see  \cite[appendix]{MS}) it follows that $$|II|+|III|+|IV|+|V| \leq C_1$$ for some constant $C_1$ depending only on $\epsilon_0$ and $g(0)$.  From this and the fact that $\Phi(t)\to \infty$ as $t\to 0$, we see that $R_{\gamma}(x, t)\to 0$ uniformly as $t\to 0$.   On the other hand, combining Lemma \ref{L1.2} part (v), the definition of $\Phi$, and again that $\inf_{t,x} u^2 >0$ yields 
$$I \geq C_2 A_0$$
 for some positive constant $C_2$ depending only on $g(0)$.  

 From \eqref{Restimate} we obtain the estimate

$$R_{\gamma} \geq \Phi(t)^{-2}(1+\ep t^2)\inv u^{-2} [C_2 A_0 - C_1]$$ from which part (i) of the Lemma follows readily.

 This completes the proof of the Lemma.
\end{proof}

\section{Extending to $\mathbb{S}^2\times[0, \infty)$: joining collar to an exterior Scharzschild region}

We can now complete the proof of Theorem \ref{T}.  

Let $g\in \ol{\scM_+}$ Consider the Riemannian ``collar" $(\mathbb{S}^2\times(0, 1], \gamma(p, t))$ constructed in Lemma \ref{L1.3} where

$$
\gamma=(1+\ep t^2) g(t)+\Phi(t)^2 u(t,\cdot)^2 dt^2.
$$
Recall that $\epsilon>0$ can be chosen arbitrarily small, $g(1)=g_*$ while $\Phi(t)$ and $u(p, t)$ are both constant functions for $t\in [1/2, 1]$.  Now for any $m>\sqrt{\text{area}(\mathbb{S}^2, g)/16\pi}$ consider the Riemannian mass-$m$ Schwarzschild manifold $(\sphere^2\times (2m,\infty), g_{S,m})$ where
$$
g_{S,m}=t^2g_*+\left(1-\frac{2m}{t}\right)\inv dt^2
$$

 Under these exact conditions it was proved in \S2 of \cite{MS} that by choosing $\epsilon$ sufficiently small, a positive scalar curvature ``bridge" can be constructed between an interior region of the collar and an exterior region of the Schwarzschild manifold to ultimately give a metric $G(p, t)$ on $\mathbb{S}^2\times[0, \infty)$ which satisfies

\begin{enumerate}
\item[(i)] $G$ has non-negative scalar curvature,
\item[(ii)] $G(p, t)=\gamma(p, t)$ for $t\in (0, 1/2]$, and
\item[(iii)] For some $T>2m$ we have $G(p, t)=g_{S, m}(p, t)$ for $t\geq T$.  
\end{enumerate}

From Lemma \ref{L1.3} and Remark \ref{singularity} we conclude that $G$ satisfies the conclusions in Theorem \ref{T}, thus completing its proof.

\section{appendix}

We construct here, a metric $g$ on $\mathbb{S}^2$ having $\lambda_1(g)=0$ thus showing the strict inclusion $\scM_+\subset \overline{\scM_+}$. Let $\D_r:=\{z\in \R^2 : |z|<r\}$ with $\D:=\D_1$. \\

\begin{lem}\label{slice}
For any $p\in \sphere^2$, there exists a coodinate chart $(U,\phi)$ containing $p$ and $v\in W^{1,2}(\sphere^2)$ such that
\begin{enumerate}
\item[(i)] $\supp(v)=U$,
\item[(ii)] $v$ is smooth on $U$, and
\item[(iii)] there exists some open $V\subset U$ such that $\ol{V}\supset \partial U$ and $\Delta_* v>0$ on $V$.  
\end{enumerate}
\end{lem}

\begin{proof} [Proof of Lemma \ref{slice}]
Let $(U_1,\phi_1)$ be a coodinate chart centered at $p$ with the property that $\phi_1(U_1)=\D$. Now define $\{(U_j, \phi_j)\}_{j=1}^\infty$  recursively as follows: Given $(U_n, \phi_n)$, let $U_{n+1}=\phi_n\inv(\D_{1/2})$ and let $\phi_{n+1}: U_{n+1}\to \R^2$ be defined as 
$$
\phi_{n+1}(x)=2\phi_n(x).
$$ 
We now observe that 
\begin{enumerate}
\item[(i)] $U_{n+1}\subset U_n$ for all $n\in\N$, 
\item[(ii)] $\phi_n(U_n)=\D$ for all $n\in\N$, and
\item[(iii)] $\cap_{n=1}^\infty U_n=\{p\}$.
\end{enumerate}
For each $n\in \N$, let $\gamma_n=(\phi_n\inv)^* g$ so that $\{\gamma_n\}$ are Riemannian metrics on $\D$. We can further require $\gamma_1(0)_{ij}=\delta_{ij}$ by possibly modifying our choice of $(U_1,\phi_1)$. One can check that $\lim_{n\to\infty} 2^{2n} \gamma_n = \gamma$ in the $C^1$ norm where $\gamma$ is the standard Euclidean metric on $\D$.

Consider $u: \D\to \R$, $x\mapsto (1-|x|)^2$ which is smooth on $\D$.  Then using polar coordinates,
$$
\Delta_\gamma u=\left(\frac{\partial^2}{\partial r^2}+\frac{1}{r}\frac{\partial}{\partial r}+\frac{1}{r^2}\frac{\partial^2}{\partial \theta^2}\right) (1-r)^2=2-\frac{2(1-r)}{r}
$$
so that $\Delta_\gamma u \geq 1$ on $\D\setminus \D_{2/3}=\{z\in \R^2 : 2/3\leq |z|<1\}$.\\

By the definition of $\Delta_{\gamma_n}$, we have 
\begin{align*}
\Delta_{\gamma_n}\big|_z&=\frac{1}{\sqrt{ \det \gamma_n(z)}} \frac{\partial}{\partial x^i}\bigg|_z\left( \gamma_n ^{ij} \sqrt{\det \gamma_n} \frac{\partial}{\partial x^j}\right)\\&=
\partial_i \gamma_n ^{ij}(z) \frac{\partial}{\partial x^j}+\frac{\gamma_n^{ij}(z) \partial_i (\det \gamma_n (z))}{2 \det \gamma_n(z)}\frac{\partial}{\partial x^j}+ \gamma_n^{ij}(z) \frac{\partial^2}{\partial x^j \partial x^i}.
\end{align*}
It follows from the $C^1$ convergence of $2^{2n}\gamma_n\to \gamma$, we can find $N\in \N$ sufficiently large enough so that $$\Delta_{\gamma_N}(u)>0\,\,\, \begin{text}on \end{text} \,\,\, \tilde{V}:=\D\setminus \overline{\D_{2/3}}.$$  Now take $U=U_N$, $v: \sphere^2\to \R: x\mapsto u(\phi_N(x)) 1_U$ (where $1_U$ is the indicator function on $U$). Then $v\in W^{1,2}(\sphere^2)$ and satisfies conditions (i), (ii), (iii) with $V=\phi_N\inv(\tilde{V})$.
\end{proof}


Let $w: \sphere^2\to \R$ be any non-constant smooth function (in particular, this means $\Delta_* w$ is not constant). Then since $\int_{\sphere^2} \Delta_* w \; d\sigma=0$, let $p\in \sphere^2$ be such that $\Delta_* w |_p >0$. Let $(U,\phi), v\in W^{1,2}(\sphere^2), V$ be as in Lemma \ref{slice} and without loss of generality, suppose $\Delta_* w \geq c>0$ on $U$ (by possibly shrinking $U$). 
Let $g_t=e^{2tAw} g_*$ with $A\geq 1/c$ to be determined. Then
\begin{align*}
\int_{\sphere^2} - v \Delta_{g_1} v +K_{g_t} v^2 \; dA_{g_t}&=\int_{U} - v \Delta_* v +(1-A\Delta_* w) v^2 \; dA_{g_*}\\&=
\int_{V} - v \Delta_* v +(1-A\Delta_* w) v^2 \; dA_{g_*}+\int_{U\setminus V} - v \Delta_* v +(1-A\Delta_* w) v^2 \; dA_{g_*}\\&\leq 
\int_{U\setminus V} - v \Delta_* v +(1-Ac) v^2 \; dA_{g_*}.
\end{align*}
Now $U\setminus V$ is closed and $v>0$ is smooth on this set. So let $0<A<\infty$ be large enough so that 
$$
A>\frac{1}{c}\left[ 1+\frac{\sup_{U\setminus V} v \Delta_*v}{\inf_{U\setminus V} v^2} \right].
$$
Then
\begin{align*}
\lambda_1(g_1)&=\inf_{\us{u\not\equiv 0}{u\in W^{1,2}(\sphere^2)}}\frac{\int_{\sphere^2} |\nabla^{g_1} u|_{g_1}^2 +K_{g_t} u^2 \; dA_{g_1}}{\int_{\sphere^2} u^2 \; dA_{g_1}}
\\&\leq\frac{\int_{\sphere^2} - v \Delta_{g_1} v +K_{g_t} v^2 \; dA_{g_1}}{\int_{\sphere^2} v^2 \; dA_{g_1}}<0
\end{align*}
So $\lambda_1(g_1)<0$ and $\lambda_1(g_0)=\lambda_1(g_*)=1$. By using the fact that the first eigenspace is one-dimensional (and thus the first eigenfunction of $L_{g_t}$ depends smoothly on $t$) and by repeated use of triangle inequality, one can check that $t\mapsto \lambda_1(g_t)$ is a continuous map. Thus, by the intermediate value theorem, there exists some $t_0\in (0,1)$ such that $\lambda_1(g_0)=0$.

\end{document}